\documentclass[11pt,reqno]{amsart}

\usepackage[left=3.5cm,right=3.5cm,top=3.5cm,bottom=3.5cm]{geometry}
\usepackage{tipa}
\usepackage{amssymb}
\usepackage{graphicx}
\usepackage[hidelinks]{hyperref}
\usepackage{enumerate}
\usepackage{xcolor}
\usepackage{amsmath}
\allowdisplaybreaks[4]

\numberwithin{equation}{section}
\newtheorem{theorem}{Theorem}[section]

\newtheorem{corollary}{Corollary}[section]
\newtheorem{lemma}[theorem]{Lemma}

\theoremstyle{definition}

\newtheorem*{remarks*}{Remarks}

\numberwithin{equation}{section}

\title{On the number of representations of $n=a+b$ with $ab$ a multiple of a polygonal number}

\author[H. Zhong]{Hao Zhong}
\address{(H. Zhong) School of Mathematical Sciences, Zhejiang University, Hangzhou, 310027, China}
\curraddr{}
\email{11435011@zju.edu.cn}
\thanks{}

\author[T. Cai]{Tianxin Cai}
\address{(T. Cai) School of Mathematical Sciences, Zhejiang University, Hangzhou, 310027, China}
\curraddr{}
\email{txcai@zju.edu.cn}
\thanks{}


\begin{document}

\maketitle

\thispagestyle{empty}

\begin{abstract}
In this paper, we study the number of representations of a positive integer $n$ by two positive integers whose product is a multiple of a polygonal number.
\end{abstract}

\section{Introduction}

Let $n$ be a positive integer. In 2014, Cai et al.\cite{Cai2014} studied the following equation

\begin{equation}\label{eq:intro}
\begin{cases}
n=a+2b, \\
ab=\binom{c}{2}
\end{cases}
\end{equation}
where $a$, $b$ an $c$ are positive integers. And they proved that (\ref{eq:intro}) is unsolvable iff $n^2+1$ is a prime. This relates to the famous $n^2+1$ problem: to show $n^2+1$ is prime infinitely, which was raised by Euler in a letter to Goldbach in 1752. It's obvious that it is true iff there exists infinite $n$ such that (\ref{eq:intro}) is unsolvable.

We denote by $P(m,c)$ the $c$th $m$-gonal number, i.e., $P(m,c)=\frac{c}{2}\{(m-2)c-(m-4)\}$. In order that $c\mapsto P(m,c)$ is injective, $c$ is a positive integer if $m=3$, $4$, and $c$ is an integer if $m>4$. In this paper, we propose a generalized function of (\ref{eq:intro}) as follows,
\begin{equation}\label{eq:main}
\begin{cases}
n=a+b, \\
ab=tP(m,c)
\end{cases}
\end{equation}
where $a$, $b$, and $t$ are positive integers. Let $r_{m,t}(n)$ denote the number of the representations of $n$ by $a+b$ with $ab=tP(m,c)$. Then we obtain a sufficient condition for $r_{m,t}(n)=0$.

\begin{theorem}\label{th:1}
If $2(m-2)n^2+t(m-4)^2$ is a prime, then (\ref{eq:main}) is unsolvable.
\end{theorem}

And if we denote by $r'_{m,t}(n)$ the number of nonnegative integer solutions to the following equation,
\begin{equation}\label{eq:mainqua}
2(m-2)n^2+t(m-4)^2=2(m-2)x^2+ty^2.
\end{equation}
Then we obtain an interesting relation between $r_{m,t}(n)$ and $r'_{m,t}(n)$.

\begin{theorem}\label{th:2}
$r_{m,t}(n)=r'_{m,t}(n)-1$ in the following three cases,\\
(1) $t=1$, and $m=3$ or $p+2$;\\
(2) $t=2$, and $m=3$ or $2p+2$;\\
(3) $t$ is an odd prime, and $t\neq m-2$, with $m=3$ or $p+2$,\\
where $p$ is an odd prime.
\end{theorem}

Let $d_{A}(n)=\prod_{p\notin A}(1+ord_{p}n)$. In particular, $d_{\varnothing}(n)=d(n)$ denoting the number of divisors of $n$. With the help of the theory of binary quadratic forms, we can obtain the results for some $m$ and $t$.

\begin{theorem}\label{th:3}
Let $ord_{p}(n)$ be the integer such that $p^{ord_{p}(n)}||n$ and $d(n)$ denote the number of divisors of $n$, where $p$ is a prime and $n$ is an integer. Then
\begin{enumerate}
\item $r_{3,1}(n)=\lfloor \{d(2n^2+1)-1\}/2 \rfloor$;
\item $r_{5,1}(n)=\lfloor \{d(6n^2+1)-1\}/2 \rfloor$;
\item $r_{7,1}(n)=\lfloor \{d_{\{3\}}(10n^2+9)-1\}/2 \rfloor$;
\item $r_{13,1}(n)=\lfloor \{d_{\{3\}}(22n^2+81)-1\}/2 \rfloor$;
\item $r_{31,1}(n)=\lfloor \{d_{\{3\}}(58n^2+729)-1\}/2 \rfloor$;
\item $r_{3,2}(n)=d(n^2+1)/2-1$;
\item $r_{8,2}(n)=\lfloor \{d_{\{2\}}(3n^2+8)-1\}/2 \rfloor$;
\item $r_{12,2}(n)=\lfloor \{d_{\{2\}}(5n^2+32)-1\}/2 \rfloor$;
\item $r_{24,2}(n)=\lfloor \{d_{\{2,5\}}(11n^2+200)-1\}/2 \rfloor$;
\item $r_{60,2}(n)=\lfloor \{d_{\{2,7\}}(29n^2+1568)-1\}/2 \rfloor$;
\item $r_{3,3}(n)=\lfloor \{d_{\{3\}}(2n^2+3)-1\}/2 \rfloor$;
\item $r_{3,5}(n)=\lfloor \{d_{\{5\}}(2n^2+5)-1\}/2 \rfloor$;
\item $r_{3,11}(n)=\lfloor \{d_{\{11\}}(2n^2+11)-1\}/2 \rfloor$;
\item $r_{3,29}(n)=\lfloor \{d_{\{29\}}(2n^2+29)-1\}/2 \rfloor$.
\end{enumerate}
\end{theorem}

Let $\mathbb{P}$ denote the set of prime numbers. Then we have
\begin{corollary}\label{cor:1}
\begin{enumerate}

\item $r_{3,1}(n)=0$ iff $2n^{2}+1\in\mathbb{P}$;

\item $r_{5,1}(n)=0$ iff $6n^{2}+1\in\mathbb{P}$;

\item $r_{7,1}(n)=0$ iff $10n^{2}+9\in\mathbb{P}\cup9\mathbb{P}$;

\item $r_{13,1}(n)=0$ iff $22n^{2}+81\in\mathbb{P}\cup9\mathbb{P}\cup81\mathbb{P}$;

\item $r_{31,1}(n)=0$ iff $58n^{2}+729\in\mathbb{P}\cup9\mathbb{P}\cup81\mathbb{P}\cup729\mathbb{P}$;

\item $r_{3,2}(n)=0$ iff $n^{2}+1\in\mathbb{P}$;

\item $r_{8,2}(n)=0$ iff $3n^{2}+8\in\mathbb{P}\cup4\mathbb{P}\cup8\mathbb{P}$;

\item $r_{12,2}(n)=0$ iff $5n^{2}+32\in\mathbb{P}\cup4\mathbb{P}\cup16\mathbb{P}\cup32\mathbb{P}$;

\item $r_{24,2}(n)=0$ iff $11n^{2}+200\in\mathbb{P}\cup4\mathbb{P}\cup8\mathbb{P}\cup25\mathbb{P}\cup100\mathbb{P}\cup200\mathbb{P}$;

\item $r_{60,2}(n)=0$ iff $29n^{2}+1568\in\mathbb{P}\cup4\mathbb{P}\cup16\mathbb{P}\cup32\mathbb{P}\cup49\mathbb{P}\cup196\mathbb{P}\cup784\mathbb{P}\cup1568\mathbb{P}$;

\item $r_{3,3}(n)=0$ iff $2n^{2}+3\in\mathbb{P}\cup3\mathbb{P}$;

\item $r_{3,5}(n)=0$ iff $2n^{2}+5\in\mathbb{P}\cup5\mathbb{P}$;

\item $r_{3,11}(n)=0$ iff $2n^{2}+11\in\mathbb{P}\cup11\mathbb{P}$;

\item $r_{3,29}(n)=0$ iff $2n^{2}+29\in\mathbb{P}\cup29\mathbb{P}$.

\end{enumerate}
\end{corollary}

And it was proven in \cite{Iwaniec1978} that there exist infinitely many integers $n$ such that $G(n)$ has at most two prime factors, where $G(n)=an^2+bn+c$ is an irreducible polynomial with $a>0$ and $c$ is odd. Hence we have

\begin{corollary}\label{cor:2}
If $(m,t)\in\{(3,1),(5,1),(7,1),(13,1),(31,1),(3,2),(3,5),(3,11),(3,29)\}$, then there are infinitely many integers $n$ such that $r_{m,t}(n)\leq1$.
\end{corollary}

\section{Preliminaries}

In order to prove the theorems, we still need some lemmas.

\begin{lemma}[{\cite[Theorem 101]{Nagell1952}}]\label{lem:1}
If $c$ and $d$ are given natural numbers, there is at most one representation of the prime $p$ in the form $p=cx^2 + dy^2$, where $x$ and $y$ are natural numbers.
\end{lemma}

To compute $r'_{m,t}(n)$, we have to apply some results in \cite{Sun2006}. We first introduce some notations of \cite{Sun2006}.

$(a,b,c):$ the binary quadratic form $ax^2+bxy+cy^2$, where $a$, $b$, $c\in \mathbb{Z}$;

$[a,b,c]:$ the equivalence class containing the form $(a,b,c)$;

$R([a,b,c],n):$ the number of representations of $n$ by the class $[a,b,c]$;

$d:$ the discriminant of $(a,b,c)$ with $d=b^2-4ac$;

$f=f(d):$ the conductor of $d$;

$H(d):$ the form class group consisting of classes of primitive, integral binary quadratic forms of discriminant $d$;

$R(S,n):=\sum_{K\in S}R(K,n)$ where $S\subseteq H(d)$;

$N(n,d):=R(H(d),n)$;

$\omega(d):=$
$$
\begin{cases}
1 & \text{ if } d>0,\\
2 & \text{ if } d<-4,\\
4 & \text{ if } d=-4,\\
6 & \text{ if } d=-3;
\end{cases}
$$

$F(A,n):=$
$$
\begin{cases}
(R(I,n)-R(A,n))/\omega(d) & \text{ if } h(d)=2,3,\\
(R(I,n)-R(A^2,n))/\omega(d) & \text{ if } h(d)=4,\\
\end{cases}
$$
where $H(d)$ is cyclic with identity $I$ and generator $A$;

$(\frac{*}{*}):$ Kronecker symbol;

$\chi(n,d):=F(A,n)\prod_{(\frac{d/f^2}{p})=1}(1+ord_{p}n)^{-1}$ where $d<0$, $h(d)=2$, $d\neq-60$ and $(n,f)=1$.

From \cite{Cox1989}, \cite{Cohen1993}, \cite{Williams1991} and \cite{Williams1994}, we have the following lemma.

\begin{lemma}\label{lem:2}
Let $d<0$ be a discriminant of some binary quadratic form. Then

(1)$h(d)=1\Leftrightarrow$ $d=$ -3, -4, -7, -8, -11, -12, -16, -19, -27, -28, -43, -67, -163.

(2)$h(d)=2\Leftrightarrow$ $d=$ -15, -20, -24, -32, -35, -36, -40, -48, -51, -52, -60, -64, -72, -75, -88, -91, -99, -100, -112, -115, -123, -147, -148, -187, -232, -235, -267, -403, -427.

(3)$h(d)=3\Leftrightarrow$ $d=$ -23, -31, -44, -59, -76, -83, -92, -107, -108, -124, -139, -172, -211, -243, -268, -283, -307, -331, -379, -499, -547, -643, -652, -883, -907.

(4)$H(d)\simeq\mathbb{Z}_{4}\Leftrightarrow$ $d=$ -39, -55, -56, -63, -68, -80, -128, -136, -144, -155, -156, -171, -184, -196, -203, -208, -219, -220, -252, -256, -259, -275, -291, -292, -323, -328, -355, -363, -387, -388, -400, -475, -507, -568, -592, -603, -667, -723, -763, -772, -955, -1003, -1027, -1227, -1243, -1387, -1411, -1467, -1507, -1555.

\end{lemma}

And we still need some more lemmas.

\begin{lemma}[Dirichlet]\label{lem:3}
When $(n,d)=1$, Dirichlet proved the following formula:
\begin{equation}\label{eq:lem3}
N(n,d)=\omega(d)\sum_{k\mid n}(\frac{d/f^2}{k})
\end{equation}
\end{lemma}

\begin{lemma}[{\cite[Remark 4.1]{Sun2006}}]\label{lem:4}
When $h(d)=1$,
\begin{equation}\label{eq:lem4}
R([1,\frac{1-(-1)^d}{2},\frac{1}{4}(\frac{1-(-1)^d}{2}-d)],n)=N(n,d).
\end{equation}
\end{lemma}

\begin{lemma}[{\cite[Theorem 9.3]{Sun2006}}]\label{lem:5}
Let $d<0$ be a fundamental discriminant and $d\neq-60$. Suppose $h(d)=2$ and $(\frac{d}{p})=0,1$ for every prime p with $2\nmid ord_{p}(n)$. Then
\begin{equation}\label{eq:lem5.1}
R(I,n)=(1+\chi(n,d))\prod_{(\frac{d}{p})=1}(1+ord_{p}n)
\end{equation}
and
\begin{equation}\label{eq:lem5.2}
R(A,n)=(1-\chi(n,d))\prod_{(\frac{d}{p})=1}(1+ord_{p}n).
\end{equation}
\end{lemma}

\begin{lemma}[{\cite{Sun2006}}]\label{lemma:6}
Let $d<0$ be a fundamental discriminant and $H(d)\simeq \mathbb{Z}_{4}$. Then
$$F(I,n)=\frac{N(n,d)}{\omega(d)};$$

$$F(A,n)=(-1)^{\mu}m\prod_{p\notin R(I)\cup R(A^2)}\frac{1+(-1)^{ord_{p}n}}{2}\cdot \prod_{p\mid m}(1-\frac{1}{p}(\frac{d}{p}))\cdot \prod_{p\in R(I)\cup R(A^2),p\nmid d}(1+ord_{p}(n))$$ where $$\mu=\sum_{p\in R(A),ord_{p}n\equiv2\pmod4}1+\sum_{p\in R(A^2),ord_{p}n\equiv1\pmod2}1;$$
and
$$F(A^2,n)=(-1)^{\sum_{p\in R(A)}ord_{p}n}\cdot F(I,n).$$
\end{lemma}

\begin{lemma}[{\cite[Theorem 11.3]{Sun2006}}]\label{lem:7}
Let $d$ be a negative fundamental discriminant and $H(d)\simeq \mathbb{Z}_{4}$. Then
\begin{equation}\label{eq:lem7.1}
R(I,n)=\omega(d)(F(I,n)+2F(A,n)+F(A^2,n))/4,
\end{equation}
\begin{equation}\label{eq:lem7.2}
R(A^2,n)=\omega(d)(F(I,n)-2F(A,n)+F(A^2,n))/4.
\end{equation}
\end{lemma}

\section{Proofs of the theorems}

\begin{proof}[Proof of Theorem \ref{th:1}]

First, we prove that if $(a,b,c)$ is an integer solution to (\ref{eq:main}), then $(|a-b|,|2(m-2)c-(m-4)|)$ is a nonnegative solution to (\ref{eq:mainqua}).
\begin{align*}
 & 2(m-2)(a-b)^2+t(2(m-2)c-(m-4))^2\\
 &= 2(m-2)((a+b)^2-4ab)+t(2(m-2)c-(m-4))^2\\
 &= 2(m-2)(n^2-4tP(m,c))+4t(m-2)c((m-2)c-(m-4))+t(m-4)^2\\
 &= 2(m-2)n^2-8t(m-2)P(m,c)+8t(m-2)P(m,c)+t(m-4)^2\\
 &= 2(m-2)n^2+t(m-4)^2.
\end{align*}

And if there is another integer solution $(a',b',c')$ such that $|a-b|=|a'-b'|$ and $|2(m-2)c-(m-4)|=|2(m-2)c'-(m-4)|$, then $a=a'$ or $b=a'$. Hence, $r_{m,t}(n)\leq r'_{m,t}(n)$. And it's obvious that $(x,y)=(n,|m-4|)$ is a solution to (\ref{eq:mainqua}). However $a$, $b>1$, so $r_{m,t}(n)\leq r'_{m,t}(n)-1$. By Lemma \ref{lem:1}, if $2(m-2)n^2+t(m-4)^2$ is a prime, then $r'_{m,t}(n)=1$ and $r_{m,t}(n)=0$, namely, (\ref{eq:main}) is unsolvable.

\end{proof}

\begin{proof}[Proof of Theorem \ref{th:2}]

Let $(x,y)$ be an integer solution to (\ref{eq:mainqua}) and $p$ be an odd prime.\\
1. Case $t=1$ and $m=3$ or $p+2$

If $n<x$, then $y<|m-4|$, which is impossible for $m=3$ and $5$. For $m>5$, $y<|m-4|$, so $0\leq y\leq m-5$. If $y=0$, then $2pn^2+(p-2)^2=2px^2$, which contradicts $p\nmid (p-2)^2$. Hence $1\leq y \leq m-5$. According to (\ref{eq:mainqua}),
\begin{equation}\label{eq:quat=1}
2(m-2)n^2+(m-4)^2=2(m-2)x^2+y^2.
\end{equation}
Hence $2(m-2)(n+x)(n-x)=(y+m-4)(y-m+4)$. And $2p=2m-4$ divides $y+m-4$ or $m-4-y$. However $m-3\leq y+m-4\leq 2m-9$ and $1\leq m-4-y\leq m-5$. Here $m-5<2m-9<2m-4$. A contradiction. Thus $n\neq x$.

In the case of $n>x$, we prove that $((n+x)/2,(n-x)/2,c)$ is an positive integer solution to (\ref{eq:main}) where $c=\frac{(m-4)+y}{2(m-2)}$ or $\frac{(m-4)-y}{2(m-2)}$ such that $c$ is an integer ($c>0$ if $m=3$). By (\ref{eq:quat=1}), $n$ and $x$ have the same parity, so $(n+x)/2$ and $(n-x)/2$ are positive integers. As we discussed in the case $n>x$, $2p=2m-4$ divides $m-4+y$ or $m-4-y$, which implies that either of $\frac{(m-4)+y}{2(m-2)}$ or $\frac{(m-4)-y}{2(m-2)}$ is an integer. Hence we can make $c$ an integer such that $y=|2(m-2)c-(m-4)|$. Now, we verify that $((n+x)/2,(n-x)/2,c)$ is a solution to (\ref{eq:main}). $(n+x)/2+(n-x)/2=n$ and $\frac{1}{2}(n+x)\cdot \frac{1}{2}(n-x)$ $=\frac{y^2-(m-4)^2}{8(m-2)}$ $=\frac{(2(m-2)c-(m-4))^2-(m-4)^2}{8(m-2)}$ $=\frac{c}{2}((m-2)c-(m-4))$ $=P(m,c)$.
Moreover, $(x,y)\mapsto((n+x)/2,(n-x)/2,c)$ is injective.

In the case of $n=x$, $(n,|m-4|)$ is an integer solution, but $((n+x)/2,(n-x)/2,c)$ can not provide a positive integer solution for (\ref{eq:main}) because $(n-x)/2=0$.

Thus, $r_{m,t}\neq r'_{m,t}-1$. And by the proof of Theorem \ref{th:1}, $r_{m,t} = r'_{m,t}-1$.\\
2. Case $t=2$ and $m=3$ or $2p+2$

If $n<x$, then $y<|m-4|$, which is impossible for $m=3$. For $m>5$, $y<|m-4|$, so $0\leq y\leq m-5$. If $y=0$, then $2pn^2+4(p-1)^2=2px^2$, which contradicts $p\nmid 4(p-1)^2$. Hence $1\leq y \leq m-5$. According to (\ref{eq:mainqua}),
\begin{equation}\label{eq:quat=2}
(m-2)n^2+(m-4)^2=(m-2)x^2+y^2.
\end{equation}
Hence $(m-2)(n+x)(n-x)=(y+m-4)(y-m+4)$. And $2p=2m-2$ divides $y+m-4$ or $m-4-y$. However $m-3\leq y+m-4\leq 2m-9$ and $1\leq m-4-y\leq m-5$. Here $m-5<2m-9<2m-4$. A contradiction. Thus $n\neq x$.

In the case of $n>x$, we proceed our discussion in two cases.

If $m=3$, then we have $n^{2}+1=x^{2}+y^{2}$. Since $n^{2}+1\equiv x^{2}+y^{2}\pmod 4$, we have $n^2\equiv x^2\pmod 4$ or $n^2\equiv y^2\pmod 4$. Transposing $x$ and $y$ if necessary, we may assume $n^2\equiv x^2\pmod 4$, which implies that $n\equiv x\pmod2$ and $y$ is odd. Now we prove that $((n+x)/2,(n-x)/2,(y-1)/2)$ is a positive integer solution to (\ref{eq:main}). $(n+x)/2+(n-x)/2=n$ and $\frac{1}{2}(n+x)\cdot \frac{1}{2}(n-x)$ $=\frac{1}{4}(y^2-1)$ $\frac{1}{4}((2c+1)^2-1)$ $=2P(3,c)$.

 Let $m=2p+2$, we prove that $((n+x)/2,(n-x)/2,c)$ is an positive integer solution to (\ref{eq:main}) where $c=\frac{(m-4)+y}{2(m-2)}$ or $\frac{(m-4)-y}{2(m-2)}$ such that $c$ is an integer. By (\ref{eq:quat=2}), $n$ and $x$ have the same parity, so $(n+x)/2$ and $(n-x)/2$ are positive integers. As we discussed in the case $n>x$, $2p=2m-4$ divides $m-4+y$ or $m-4-y$, which implies that either of $\frac{(m-4)+y}{2(m-2)}$ or $\frac{(m-4)-y}{2(m-2)}$ is an integer. Hence we can make $c$ an integer such that $y=|2(m-2)c-(m-4)|$. Now, we verify that $((n+x)/2,(n-x)/2,c)$ is a solution to (\ref{eq:main}). $(n+x)/2+(n-x)/2=n$ and $\frac{1}{2}(n+x)\cdot \frac{1}{2}(n-x)$ $=\frac{y^2-(m-4)^2}{4(m-2)}$ $=\frac{(2(m-2)c-(m-4))^2-(m-4)^2}{4(m-2)}$ $=c((m-2)c-(m-4))$ $=2P(m,c)$.
Moreover, $(x,y)\mapsto((n+x)/2,(n-x)/2,c)$ is injective.

In the case of $n=x$, $(n,|m-4|)$ is an integer solution, but $((n+x)/2,(n-x)/2,c)$ can not provide a positive integer solution for (\ref{eq:main}) because $(n-x)/2=0$.

Thus, $r_{m,t}\neq r'_{m,t}-1$. And by the proof of Theorem \ref{th:1}, $r_{m,t} = r'_{m,t}-1$.\\
3. Case $t$ is an odd prime and $t\neq m-2$, with $m=3$ or $p+2$

It's similar to the proof of the above two cases. Since we can verify that $((n+x)/2,(n-x)/2,c)$ is a positive integer solution to (\ref{eq:main}) where $c=\frac{(m-4)+y}{2(m-2)}$ or $\frac{(m-4)-y}{2(m-2)}$ such that $c$ is an integer ($c>0$ if $m=3$) and the case $n\leq x$ can only provide one solution $(n,|m-4|)$, we have $r_{m,t} = r'_{m,t}-1$.
\end{proof}

\begin{proof}[Proof of Theorem \ref{th:3}]

Let $\mathbb{S}$ be the set of square numbers. By the definition of $r'_{m,t}(n)$, we have
\begin{align*}
&r'_{m,t}(n)=\\
&\begin{cases}
\frac{1}{4}R([t,0,2(m-2)],2(m-2)n^2+t(m-4)^2)& \text{if } 2(m-2)n^2+t(m-4)^2 \notin \mathbb{S};\\
\frac{1}{4}\{R([t,0,2(m-2)],2(m-2)n^2+t(m-4)^2)+2\}& \text{if } 2(m-2)n^2+t(m-4)^2 \in \mathbb{S}.
\end{cases}
\end{align*}

Under the condition of Theorem \ref{th:2}, we can obtain the value of $r_{m,t}(n)$ by applying the lemmas about the binary quadratic forms.\\
1. Case $h(d)=1$ ($d=-4$ or $-8$)

If $(m,t)\in \{(3,1),(3,2)\}$, then $h(d)=1$. To make this paper brief, we omit the proof of $d=-4$ which is similar to the case $d=-8$. Let $d=-8$. Then $(m,t)=(3,1)$. Applying Lemma \ref{lem:3} and Lemma \ref{lem:4}, we have
\begin{align*}
&r_{3,1}(n)=r'_{3,1}(n)-1\\=
&\begin{cases}
\frac{1}{4}R([1,0,2],2n^2+1)-1& \text{if } 2n^2+1 \notin \mathbb{S};\\
\frac{1}{4}\{R([1,0,2],2n^2+1)+2\}-1& \text{if } 2n^2+1 \in \mathbb{S}.
\end{cases}
\\=
&\begin{cases}
\frac{1}{4}N(2n^2+1,-8)-1& \text{if } 2n^2+1 \notin \mathbb{S};\\
\frac{1}{4}\{N(2n^2+1,-8)+2\}-1& \text{if } 2n^2+1 \in \mathbb{S}.
\end{cases}
\\=
&\begin{cases}
\frac{1}{2}\sum_{k\mid 2n^2+1}(\frac{-2}{k})-1& \text{if } 2n^2+1 \notin \mathbb{S};\\
\frac{1}{2}\sum_{k\mid 2n^2+1}(\frac{-2}{k})-\frac{1}{2}& \text{if } 2n^2+1 \in \mathbb{S}.
\end{cases}
\end{align*}
And we notice that for any prime divisor $q$ of $2n^2+1$, we have $2n^2+1\equiv 0\pmod q$. It follows that $(\frac{-2}{q})=1$ and
\begin{align*}
r_{3,1}(n)&=
\begin{cases}
\frac{1}{2}\prod_{p\mid 2n^2+1}(1+ord_{p}(2n^2+1))-1& \text{if } 2n^2+1 \notin \mathbb{S};\\
\frac{1}{2}\prod_{p\mid 2n^2+1}(1+ord_{p}(2n^2+1))-\frac{1}{2}& \text{if } 2n^2+1 \in \mathbb{S}.
\end{cases}
\\&=
\lfloor \{d(2n^2+1)-1\}/2 \rfloor.
\end{align*}
\\2. Case $h(d)=2$ ($d=-24$, $-40$, $-880$ or $-232$)

If $(m,t)\in \{(5,1),(7,1),(13,1),(31,1),(8,2),(12,2),(24,2),(60,2),(3,3),(3,5),(3,11)$, $(3,29)\}$, then $h(d)=2$. We only give the proof of $(m,t)=(5,1)$ for the others are similar. As an application of Lemma \ref{lem:5}, we have
\begin{align*}
r_{5,1}(n)&=
\begin{cases}
\frac{1}{4}(1+(-1)^{ord_{3}(6n^2+1)}(\frac{6n^2+1}{3}))\prod_{p}(1+ord_{p}(6n^2+1))-1& \text{if } 6n^2+1 \notin \mathbb{S};\\
\frac{1}{4}(1+(-1)^{ord_{3}(6n^2+1)}(\frac{6n^2+1}{3}))\prod_{p}(1+ord_{p}(6n^2+1))-\frac{1}{2}& \text{if } 6n^2+1 \in \mathbb{S}.
\end{cases}
\\&=
\begin{cases}
\frac{1}{2}\prod_{p}(1+ord_{p}(6n^2+1))-1& \text{if } 6n^2+1 \notin \mathbb{S};\\
\frac{1}{2}\prod_{p}(1+ord_{p}(6n^2+1))-\frac{1}{2}& \text{if } 6n^2+1 \in \mathbb{S}.
\end{cases}
\\&=
\lfloor \{d(6n^2+1)-1\}/2 \rfloor.
\end{align*}
\end{proof}

\section{The Case $H(d)\simeq\mathbb{Z}_{4}$}
Applying Lemma \ref{lem:7}, we have some complex results.
\begin{theorem}\label{th:4}
Let $d$ be a negative fundamental discriminant such that $H(d)\simeq\mathbb{Z}_{4}$.

(1) Let $z=2(m-2)n^2+(m-4)^2$. If $(m,t)\in\{(9,1),(19,1),(25,1),(43,1),(73,1)\}$, then
\begin{equation}\label{th:4.1}
r_{m,t}(n)=
\begin{cases}
\frac{1}{8}(F(I,z)+2F(A,z)+F(A^2,z))-1& \text{if } z \notin \mathbb{S};\\
\frac{1}{8}(F(I,z)+2F(A,z)+F(A^2,z))-\frac{1}{2}& \text{if } z \in \mathbb{S}.
\end{cases}
\end{equation}

(2) Let $z=\{(m-2)n^2+(m-1)^2\}/2$. If $(m,t)\in\{(16,2),(36,2),(48,2),(84,2),(144,2)\}$, then
\begin{equation}\label{th:4.2}
r_{m,t}(n)=
\begin{cases}
\frac{1}{8}(F(I,z)+2F(A,z)+F(A^2,z))-1& \text{if } 2z \notin \mathbb{S};\\
\frac{1}{8}(F(I,z)+2F(A,z)+F(A^2),z)-\frac{1}{2}& \text{if } 2z \in \mathbb{S}.
\end{cases}
\end{equation}
\end{theorem}

The following table indicates the values of $F(I)$ in Theorem \ref{th:4}.

\begin{tabular}{|c|c|c|c|c|c|c|c|}
  \hline
  $m$ & $t$ & $d$ & $z$ & $I$ & $A$ & $A^2$ & $F(I)$ \\
  9 & 1 & -56 & $14n^2+25$ & $[1,0,14]$ & $[3,2,5]$ & $[2,0,7]$ & $d(z)$ \\
  19 & 1 & -136 & $34n^2+225$ & $[1,0,34]$ & $[5,2,7]$ & $[2,0,17]$ & $d_{\{3\}}(z)$ \\
  25 & 1 & -184 & $46n^2+441$ & $[1,0,46]$ & $[5,4,10]$ & $[2,0,23]$ & $d_{\{3,7\}}(z)$ \\
  43 & 1 & -328 & $82n^2+1521$ & $[1,0,82]$ & $[7,6,13]$ & $[2,0,41]$ & $d_{\{3\}}(z)$ \\
  73 & 1 & -568 & $142n^2+4761$ & $[1,0,142]$ & $[11,2,13]$ & $[2,0,71]$ & $d_{\{3,23\}}(z)$ \\
  16 & 2 & -56 & $7n^2+72$ & $[1,0,14]$ & $[3,2,5]$ & $[2,0,7]$ & $d_{\{2\}}(z)$ \\
  36 & 2 & -136 & $17n^2+512$ & $[1,0,34]$ & $[5,2,7]$ & $[2,0,17]$ & $d_{\{2\}}(z)$ \\
  48 & 2 & -184 & $23n^2+968$ & $[1,0,46]$ & $[5,4,10]$ & $[2,0,23]$ & $d_{\{2\}}(z)$ \\
  84 & 2 & -328 & $41n^2+3200$ & $[1,0,82]$ & $[7,6,13]$ & $[2,0,41]$ & $d_{\{2,5\}}(z)$ \\
  144 & 2 & -568 & $71n^2+9800$ & $[1,0,142]$ & $[11,2,13]$ & $[2,0,71]$ & $d_{\{2,5,7\}}(z)$ \\
  \hline
\end{tabular}

\subsection*{Acknowledgments}
This work is supported by the National Natural Science Foundation of China (Grant No. 11501052 and Grant No. 11571303).

\bibliographystyle{amsplain}

\end{document}